\documentclass[11pt,a4paper]{amsart}

\usepackage{graphicx}
\usepackage[font=small,labelfont=bf]{caption}
\usepackage{amssymb}

\newtheorem{theorem}{Theorem}[section]
\newtheorem{lemma}[theorem]{Lemma}
\newtheorem{proposition}[theorem]{Proposition}

\newtheorem{thmdf}[theorem]{Theorem-Definition}
\theoremstyle{definition}
\newtheorem{definition}[theorem]{Definition}
\theoremstyle{remark}
\newtheorem{remark}[theorem]{Remark}
\newtheorem*{example}{Example}

\newcommand{\N}{\mathbb{N}}

\newcommand{\Q}{\mathbb{Q}}

\newcommand{\A}{\mathbb{A}}

\def\co{\colon\thinspace}

\begin{document}

\title[Computing local intersection multiplicity]
{Computing local intersection multiplicity of plane curves via blowup}


\author{Jana Chalmoviansk\'a, Pavel Chalmoviansk\'y}
\thanks{This work was supported by the Slovak Research and Development Agency 
under the contract No.~APVV-16-0053}
\address{Faculty of Mathematics, Physics and Computer Science,
         Comenius University, Bratislava, Slovakia}
\email{jana.chalmovianska@fmph.uniba.sk, pavel.chalmoviansky@fmph.uniba.sk}
% \urladdr{www.math.sc.edu/$\sim$howard} % Delete if not wanted.

\begin{abstract}
  We prove that intersection multiplicity of two plane curves defined
  by Fulton's axioms is equivalent to the multiplicity computed using blowup.
  The algorithm based on the latter is presented and 
  its complexity is estimated.
  We compute for polynomials over $\Q$ and its algebraic extensions.
\end{abstract}

\maketitle

% \begin{keyword}
% central simple algebra \sep zero divisor
% \end{keyword}

\section{Introduction}

The classical result in algebraic geometry on plane curves, 
B\'ezout's theorem, states that the number of intersections of two 
curves with no common component equals to the product of their degrees 
provided 
\begin{itemize}
  \item the curves are defined in projective plane,
  \item the intersections are computed over an algebraically closed field,
  \item each intersection point is counted with proper multiplicity.
\end{itemize}
The most obscure part is the computation of the multiplicity 
of the intersection in a particular point,
which is generally a challenging task from both computational
and interpretation point of view (\cite{flenner_carroll_vogel, boda_schenzel}), 
i.e.~its intersection number. 

During the development of the subject, 
several definitions of the intersection number for two curves
in a given point were formulated. Nowadays, the definition by~\cite{fulton} 
is probably the most known and accepted.
We cite it in Section~\ref{se:intersection_number}. 
The definition gives already an algorithm for computing 
the intersection number
and it was implemented in Magma %~\cite{magma}
by~\cite{hilmar_smyth}. 
Their algorithm lists all points of
intersection of two algebraic curves, together with their multiplicities.

In~\cite{wall}, the intersection number of two curves in a given point
is described as the number of intersections that appear instead of
the given one after we wiggle the curves a little bit. If one of the
curves is a line (or more generally a rational curve), 
the intersection number can be easily computed: 
if $(\varphi(t),\psi(t))$ is a parametrization of one curve, 
we plug it into the polynomial $g$ defining the other curve
and then the multiplicity of intersection in the point 
$(\varphi(t_0),\psi(t_0))$ is the multiplicity of the root $t_0$ 
in the equation $g(\varphi(t),\psi(t)) = 0$.
In case none of the curves is rational, 
the parametrization of branches by Puiseux series can be used.

Alternatively, the intersection multiplicity of two curves can 
be computed using resultants (\cite{gibson,walker}). 
This is proven to be equivalent to the intersection number 
given by Fulton~(\cite{sendra_winkler}).

The geometric meaning of intersection multiplicity is expressed
by relating it to the infinitely near points of the curve.
For example, in a point in common for two curves,
sharing a first order infinitely near point 
corresponds to sharing a tangent line 
and sharing also a second order infinitely near point 
corresponds to sharing an osculating circle.
The connection between the intersection number and the shared
infinitely near points is well studied in~\cite{wall} using
Puiseux series
or in~\cite{zariski} using valuations.

The infinitely near points are looked for 
using birational morphism of the plane called blowup, 
which we briefly explain in Section~\ref{se:blowup}.
In the paper, we give a proof that the number computed by counting 
the shared infinitely near points with their multiplicities
is the same as the intersection number defined by Fulton,
and this is the main result of the paper:

\noindent{\rm\textbf{Theorem \ref{thm:main}.}}
{\it  Let $f,g\in k[x,y]$ be non-constant polynomials
  and $P\in\A^2(k)$ be a point.
  Then $$\mathcal{B}_P(f,g) = I_P(f,g),$$
  where $\mathcal{B}_P(f,g)$ is the intersection number computed  using 
  infinitely near points at $P$ common for curves defined by $f$ and $g$,
  and $I_P(f,g)$ is the intersection number of the curves 
  defined by Fulton.}

So we can use the infinitely near points 
when computing the intersection number.
When comparing with the algorithm by~\cite{hilmar_smyth},
the proposed algorithm  
computes the intersection multiplicity only in one point. 
But the tests show that in case the intersection multiplicity 
in the given point is high 
(i.e. the point is singular for one or both curves, 
or the curves share more geometric invariants in the point), 
or in case the curves themselves are of high degree,
our algorithm turns out to be more effective.
The performance of the algorithm is discussed at the end of the paper.

\section{Preliminaries}
\subsection{Notation and terminology}

The assertions in the paper are proven under the assumptions that
the field $k$ is algebraically closed and its characteristic is 0.

In the paper, we work in the affine plane over the field $k$.
A curve in $\A^2$ is defined by 
a single non-constant polynomial from $k[x,y]$.
A curve defined by a polynomial $f$ is denoted $C_f$.
The {\em points on/of a curve $C_f$} are 
all roots of $f$ with coordinates in $k$.
The set of all points of a curve $C_f$ we denote by $V(f)$.
We will distinguish e.g. a curve defined by $2y - x^2$ and 
a curve defined by $(2y - x^2)^2$, but we will not distinguish
curves defined by $2y - x^2$ and $3x^2 - 6y$. 
So there is a bijection between plane curves defined over $k$
and proper principal ideals $(f)$ in $k[x,y]$.

Let $C_f$ be a plane curve and let $P = (p_x,p_y)\in\A^2$ be a point.
We can write 
$$f = a_{00} + a_{10}(x-p_x) + a_{01}(x-p_y) + 
      a_{20}(x-p_x)^2 + a_{11}(x-p_x)(x-p_y) + a_{02}(x-p_y)^2 + \dots$$
as a polynomial in $x-p_x$, $y-p_y$ (the Taylor series of $f$ at $P$).
Then the {\em multiplicity of $C_f$ at $P$}
is the degree of the lowest term with non-zero coefficients 
of the Taylor series of $f$ at $P$ with nonzero coefficient, 
we denote it by $m_P(C_f)$ or $m_P(f)$.
In case $P = (0,0)$, we may shorten the notation as $m(C_f)$ or $m(f)$.

If $\varphi\co\A^2\to\A^2$ is an affine transform of the plane,
then $f^\varphi$ denotes the polynomial $f\circ\varphi^{-1}$ 
i.e. $f^\varphi$ is a polynomial describing 
the transform of the curve $C_f$. 
In particular, the points of $C_{f^\varphi}$ are exactly 
the images of the points of $C_f$ under the transform $\varphi$:
$$\varphi(P)\in V(f^\varphi)\quad\text{if and only if}\quad P\in V(f).$$

\subsection{Blowup}~\label{se:blowup}

Blowing up~(\cite{shafarevich, cutkosky}) 
is a powerful tool for resolving singularities of plane curves.
However a curve can be blown up also in a regular point and we use it for
computing the intersection multiplicity of the curves.

When blowing up the affine plane $\A^2$ in a point $P=(p_1,p_2)$, 
we construct a surface $X\subset\A^2\times\mathbb{P}^1$ 
given by $X = V((x-p_1)t_1-(y-p_2)t_0)$,
where $(x,y)$ are the coordinates in $\A^2$ and $(t_0\co t_1)$ are the homogeneous
coordinates in $\mathbb{P}^1$. {\em The blowup} of $\A^2$ is the surface $X$
together with the birational morphism 
$$\pi\co X\to\A^2,\quad (x,y;t_0\colon t_1)\mapsto(x,y).$$

If $C_f$ is a curve in $\A^2$, 
its blowup is the preimage of $C_f$ under $\pi$, 
i.e. it is the curve on $X$ defined by the polynomial $\pi^*f = f\circ\pi$.
The blowup of $C_f$ consists of {\em the strict transform}, 
which geometrically is the closure of $\pi^{-1}(V(f)\setminus\{P\})$ 
and algebraically it is the curve on $X$ defined 
by the saturation ideal $(\pi^*f):(x-p_1,y-p_2)^\infty$,
and {\em the exceptional line}, which is the preimage of $P$.
We denote the strict transform of $C_f$ by $C_f^\prime$.
The points where the strict transform $C_f^\prime$ meets the exceptional line
are the {\em first order infinitely near points} 
of the curve $C_f$ at the point $P$.

Let the curves $C_f$ and $C_g$ both pass through a point $P$ 
and let us consider their blowups.
Each first order infinitely near point at $P$ shared by both curves
corresponds to a tangent at $P$ that the curves $C_f$ and $C_g$ 
have in common.

When working with the strict transform of a curve $C_f$
contained in $X$,
we pass to an affine chart of $X$ isomorphic to $\A^2$, 
for example the chart with $t_0\ne 0$. 
Then the strict transform $C_f^\prime$ is described by one
polynomial and we will denote it $f^\prime$,
so locally $C_f^\prime = C_{f^\prime}$.

\begin{proposition}\label{prop:isomorphic_blowups}
  Let $\varphi$ be an affine change of coordinates,
  which takes the curve $C_f$ to a curve $C_g$.
  Then $\varphi$ induces a linear map that takes
  the strict transform of $C_f$ after blowing up in $P\in V(f)$
  to the strict transform of $C_g$ after blowing up in $\varphi(P)$.
\end{proposition}

\begin{remark}\label{rem:isomorphic_blowups}
  When blowing up a curve $C_f$ at $P$, 
  by Proposition~\ref{prop:isomorphic_blowups}, we may assume $P = (0,0)$. 
  Then we can write 
  \begin{equation}\label{eq:curve_taylor}
    f = F + f_1 
  \end{equation}
  with $F$ being the form of degree $m(f)$
  and $f_1$ containing the terms in $f$ of higher degree.
  The tangents of $C_f$ at $P=(0,0)$ correspond 
  to the linear factors of $F$.
  The blowup of $C_f$ lays on the surface 
  $X = V(xt_1-yt_0)\subset\A^2\times\mathbb{P}^1$.
  By Proposition~\ref{prop:isomorphic_blowups}, 
  we may also assume that the $y$-axis is not tangent 
  to $C_f$ at $(0,0)$, 
  so $x$ is not a factor of $F$ in~(\ref{eq:curve_taylor}).
  Then all first order infinitely near points  
  of $C_f$ at $(0,0)$ 
  are contained in the affine chart of $X$ with $t_0\ne 0$.
  Hence, the blowup of $C_f$ is locally described by
  \begin{equation}\label{eq:blowup_curve}
    f(x,xz) = x^{m(f)}(F(1,z) + xf_2(x,z))\quad
    \textrm{ for some }\quad f_2\in k[x,z],
  \end{equation}
  where $z=t_1/t_0$. 
  As usually done, we replace $z$ by $y$ in~(\ref{eq:blowup_curve}), 
  so $x^{m(f)}$ corresponds to the exceptional line 
  including its multiplicity $m(f)$ and 
  $$f^\prime(x,y) = F(1,y) + xf_2(x,y)$$ 
  is a polynomial defining locally the strict transform $C_f^\prime$
  with all exceptional points of $C_f$ at $(0,0)$ contained 
  in the considered affine chart.
\end{remark}

Let $C_f$ be a curve and $P\in V(f)$ be a point. 
After blowing up $C_f$ at $P$,
we obtain $C_f^\prime$ containing the first order infinitely near points
of $C_f$ at $P$. We can continue blowing up $C_f^\prime$
at such a point and obtain a transform of $C_f^\prime$
containing the first order infinitely near points of $C_f^\prime$ 
at the considered point, 
hence they are the {\em second order infinitely near points} of $C_f$ at $P$. 
Following the pattern, we define the {\em infinitely near points} 
of $C_f$ at $P$ {\em of order $r$} for any $r\in\N^+=\N\setminus\{0\}$.
In this way, we actually obtain a rooted tree of infinitely near points 
of $C_f$ at $P$. The root corresponds to the point $P$ and all other 
vertices correspond to infinitely near points of $C_f$ at $P$,
where the direct descendants corresponds to all the different first order 
infinitely near points of the considered point.
The tree is called the {\em configuration} of infinitely near points 
of $C_f$ at $P$.
By the {\em length of the configuration} we refer
to the number of consecutive blowups 
computed in order to find the configuration.
% In case the point $P$ is a regular point of $C_f$, the configuration is just
% a linear chain of points. 

\begin{proposition}\label{prop:blowup_reg}
  Let $C_f,C_g$ be plane curves and
  let $P$ be a regular point of both $C_f$ and $C_g$.
  If for any $n\in\N^+$ the configurations 
  of infinitely near points of length $n$
  of $C_f$ and $C_g$ coincide, 
  then $C_f$ and $C_g$ share a common component through $P$.
\end{proposition}

\begin{proof}
  \cite{lipman}, Theorem 2.1.
\end{proof}

\subsection{Local intersection number}\label{se:intersection_number}

We recall the basic knowledge and the definition 
of intersection of plane curves by~\cite{fulton}.

\begin{thmdf}\label{thm:fulton}
  Let $C_f,C_g$ be plane curves and let $P\in\A^2(k)$ be a point.
  There is a unique {\em intersection number $I_P(f,g)$} 
  defined for all plane curves $C_f,C_g$ and all points $P\in\A^2(k)$ 
  satisfying the following properties:
  \begin{enumerate}
    \item If $C_f,C_g$ have no common component passing through $P$,
      then $I_P(f,g)$ is a non-negative integer; 
      otherwise $I_P(f,g) = \infty$.
    \item $I_P(f,g) = 0$ if and only if $P\notin V(f)$ or $P\notin V(g)$.
    \item If $\varphi$ is an affine change of coordinates, then 
      $I_P(f,g) = I_{\varphi(P)}(f^\varphi,g^\varphi)$,
%       where $f^\varphi,g^\varphi$ are the polynomials of the curves
%       after applying the transform $\varphi$.
    \item $I_P(f,g) = I_P(g,f)$.
    \item $I_P(f,g) \ge m_P(f)m_P(g)$, where the equality occurs 
      if and only if $C_f$ and $C_g$ have no common tangent line at $P$
      (so called {\em transversal intersection}).
    \item $I_P(f_1f_2,g) = I_P(f_1,g) + I_P(f_2,g)$
      for any $f_1,f_2\in k[x,y]$.
    \item $I_P(f,g) = I_P(f,g+hf)$ for any $h\in k[x,y]$.
  \end{enumerate}
\end{thmdf}

\section{Local properties of blowups of intersection.}

\begin{definition}\label{def:B}
  Let $C_f, C_g$ be plane curves and let $P\in\A^2$ be a point.
  Let $C_f^\prime, C_g^\prime$ be the strict transforms of $C_f,C_g$ respectively 
  under the blowing up the plane at the point $P$.
  Then, we define the number 
  \begin{displaymath}
  \mathcal{B}_P(f,g) = \begin{cases}
    \infty,\quad\text{ if }C_f\text{ and }C_g
      \text{ share a component passing through }P,\\
    0,\quad\text{ if }P\notin V(f)\cap V(g),\\
    m_P(f)m_P(g) + \sum_Q\mathcal{B}_Q(f^\prime, g^\prime)
      \quad\text{ otherwise}, 
    \end{cases}
  \end{displaymath}
  where the sum in the last case runs through 
  all the first order infinitely near points $Q$ 
  common to $C_f$ and $C_g$.
  The number $\mathcal{B}_P(f,g)$ may sometimes be denoted also
  by $\mathcal{B}_P(C_f,C_g)$.
\end{definition}

\begin{lemma}\label{lem:blowup_stop}
  Let $C_f, C_g$ be curves 
  having no common component passing through $P$.
  Then the computation of $\mathcal{B}_P(f,g)$ according to
  Definition~\ref{def:B} terminates after finitely many steps.
  More precisely, there exist $r\in\N^+$ such that 
  the curves $C_f$ and $C_g$ have no infinitely near point 
  at $P$ of order $r$ in common.
\end{lemma}

\begin{proof}
  By blowing up $C_f$ and $C_g$ in $P$ 
  and tracking the common infinitely near points at $P$, 
  in each branch either the computation stops because 
  there are no common infinitely near points, 
  or after finitely many steps we arrive to the 
  situation that we need to compute $\mathcal{B}_P(f,g)$,
  where $P$ is regular in both $C_f$ and $C_g$~(\cite{wall}),
  Theorem 3.4.4.

  So assume now that $P$ is regular in $C_f$ and $C_g$ and
  both curves share the common tangent at $P$.
  Assume that after blowing up at $P$ their strict
  transforms again share a tangent in the common
  infinitely near point, and that this situation reappears
  in each step. Then by Proposition~\ref{prop:blowup_reg} 
  the components of $(f)$ and $(g)$ through $P$ coincide, 
  a contradiction.
\end{proof}

\begin{remark}\label{rem:blowup_induction}
  We use Lemma~\ref{lem:blowup_stop} in the proofs of the
  following propositions about $\mathcal{B}$ as follows:
  for given $P$, $C_f$ and $C_g$, we want to prove a claim 
  about $\mathcal{B}_P(f,g)$. 
  If $P\notin V(f)$ or $P\notin V(g)$ or $P\in V(f)\cap V(g)$ but
  $C_f$ and $C_g$ intersect transversally in $P$,
  we prove the claim about $\mathcal{B}$ directly 
  (the start of the induction). 
  In other cases we will prove it, 
  provided the claim holds for $\mathcal{B}_Q(C_f^\prime, C_g^\prime)$,
  where $C_f^\prime$ and $C_g^\prime$ are the strict transforms 
  of $C_f$ and $C_g$ after blowing up at $P$, and $Q$ is 
  a common first order infinitely near point at $P$.
  So the induction step is to prove the claim for $\mathcal{B}_P(f,g)$
  provided the claim is true for $\mathcal{B}_Q(\tilde f,\tilde g)$,
  where the configurations of the infinitely near points 
  of $\tilde f$ resp. $\tilde g$ at $Q$ 
  (see the commentary before Proposition~\ref{prop:blowup_reg}) 
  have smaller length than those of $f$ resp. $g$ at $P$.
  We refer to this proving style as the blowup induction.
\end{remark}

In the rest of the section we will prove that 
the number $\mathcal{B}_P(f,g)$ computed 
recursively as in Definition~\ref{def:B} 
is exactly the intersection number $I_P(f,g)$.
We do it by verifying that $\mathcal{B}_P(f,g)$ satisfies the properties 
in Definition~\ref{thm:fulton}.

\begin{proposition}\label{prop:aff_transf}
  If $\varphi$ is an affine change of coordinates, then 
  $$\mathcal{B}_P(f,g) = \mathcal{B}_{\varphi(P)}(f^\varphi,g^\varphi).$$
\end{proposition}

\begin{proof}
  Affine change of coordinates taking 
  $C_f$ to $C_{f^\varphi}$ and $C_g$ to $C_{g^\varphi}$ 
  induces locally an affine change of coordinates taking 
  $C_f^\prime$ to $C_{f^\varphi}^\prime$ and $C_g^\prime$ to $C_{g^\varphi}^\prime$
  (Proposition~\ref{prop:isomorphic_blowups}).
  So the infinitely near points of $C_f$ are mapped
  to the infinitely near points of $C_{f^\varphi}$;
  the same holds for $C_g$ and $C_{g^\varphi}$.

  Further, the affine change preserves the multiplicity of a curve
  in a point, so $m_{\varphi(P)}(f^\varphi) = m_P(f)$.
  Hence the computation of $\mathcal{B}_{\varphi(P)}(f^\varphi,g^\varphi)$
  is the same as the computation of $\mathcal{B}_P(f,g)$.
\end{proof}

\begin{proposition}\label{prop:product}
  Let $C_f, C_g$ have no common component passing through $P$ 
  and let $f = uv$. Then
  $$\mathcal{B}_P(f,g) = \mathcal{B}_P(u,g) + \mathcal{B}_P(v,g).$$
\end{proposition}

\begin{proof}
  We note that the set of the first order infinitely near points 
  of $C_f = C_{uv}$ at $P$
  is the union of the sets of the first order infinitely near points 
  of $C_u$ and $C_v$ at $P$.
  For, the first order infinitely near points depend
  only on the lowest degree form of the polynomial defining the curve,
  and the lowest degree form of $f$ is the product
  of those of $u$ and $v$.

  We proceed by blowup induction, see~Remark~\ref{rem:blowup_induction}.
  First, let the curves $C_f$ and $C_g$ have 
  no common first order infinitely near point at $P$. 
  Then the pairs $C_u,C_g$ and $C_v,C_g$ also
  have no common infinitely near point at $P$. 
  Therefore
  \begin{displaymath}
    \mathcal{B}_P(f,g) = m_P(f)m_P(g) = m_P(u)m_P(g) + m_P(v)m_P(g) =
    \mathcal{B}_P(u,g) + \mathcal{B}_P(v,g).
  \end{displaymath}

  In a general case, we may assume $P = (0,0)$
  and the $y$-axis in neither tangent to $C_u$, $C_v$ nor to $C_g$ at $P$,
  by Proposition~\ref{prop:aff_transf}.
  Then one checks easily
  (see Remark~\ref{rem:isomorphic_blowups})
  that after blowing up $C_f$ in $P$ 
  we have $(uv)^\prime = u^\prime v^\prime$
  for the polynomials defining the strict transforms $C_f^\prime$. 
  Again, since the first order infinitely near points of $C_u$ and $C_v$ are all
  among the first order infinitely near points of $C_f = C_{uv}$, it holds
  \begin{eqnarray}
    \mathcal{B}_P(f,g) &=& m_P(f)m_P(g) +
      \textstyle{\sum_Q\mathcal{B}_Q((uv)^\prime,g^\prime)}\notag\\
      &=& (m_P(u)+m_P(v))m_P(g) 
          + \textstyle{\sum_Q\mathcal{B}_Q(u^\prime v^\prime,g^\prime)},\label{eq:5}
  \end{eqnarray}
  where the sum runs through the first order infinitely near points at $P$ 
  that are shared by $C_{uv}$ and $C_g$.
  Now, we can continue by induction 
  \begin{eqnarray*}
    (\ref{eq:5}) &=& m_P(u)m_P(g) + m_P(v)m_P(g) + 
      \textstyle{\sum_Q\mathcal{B}_Q(u^\prime,g^\prime)} + 
      \textstyle{\sum_Q\mathcal{B}_Q(v^\prime,g^\prime)}\\
    &=& \mathcal{B}_P(u,g) + \mathcal{B}_P(v,g).
  \end{eqnarray*}
\end{proof}

We still need to verify the last property of $\mathcal{B}_P$.
% namely that the result of the computation does not depend on
% the generators of the ideal $(f,g)$. 
To do this, we first study a special kind of generators of an ideal $I$.

\begin{definition}
  Let $I = (f,g)\subset k[x,y]$ be an ideal.
  We say that $f,g$ is a {\em max-order basis} of $I$, if
  $m(f)\le m(g)$ 
  and for every $\tilde{g}\in I$ such that $(f,\tilde{g}) = (f,g)$ 
  we have that $m(\tilde{g})\le m(g)$.
\end{definition}

\begin{example}
  A max-order basis of the ideal is not unique:
  for example $x, y^2$ is a max-order basis of the ideal they generate,
  but also $x, y^2+x^3$ is a max-order basis of the same ideal.
  On the other hand, not every ideal has a max-order basis,
  for example 
  $$(x^2-x,xy) = (x^2-x, x^2y) = (x^2-x, x^3y) = \dots.$$
\end{example}

\begin{lemma}\label{lem:mmo_exists}
  Let $C_f,C_g$ be curves with no common component through $(0,0)$.
  Then $(f,g)$ has a max-order basis.
  Moreover, if $m(f)\le m(g)$ then there exists $h\in k[x,y]$
  such that $f, g+hf$ is a max-order basis.
\end{lemma}

\begin{proof}
  Consider the ideal $(f,g)$. By $F$ we denote the lowest degree form of $f$
  and similarly let $G$ be the lowest degree form of $g$.

  It holds that $f,g$ is a max-order basis of $(f,g)$ with $m(f)\le m(g)$ 
  if and only if the $G$ is not divisible by $F$.

  So if $f,g$ is not a max-order basis of $(f,g)$, 
  then $G$ factors as $G=FH$, $H$ being a form.
  We replace $g$ by $g+Hf$ and get a new basis of the ideal:
  $(f,g) = (f,g+Hf)$ with $m(g+Hf) > m(g)$.
  The process stops, for otherwise there would be a polynomial $\tilde g$ 
  with $m(\tilde g) > \deg(f)\deg(g)$ such that $(f,\tilde g) = (f,g)$,
  which would be a contradiction to B\'ezout's theorem.
  Hence, we arrive to $g+hf$ such that $(f,g+hf) = (f,g)$ 
  and $f,g$ is a max-order basis, after finitely many steps . 
\end{proof}

\begin{proposition}\label{prop:gen_indep}
  Let $C_f, C_g$ have no common component passing through $P$.
  Then
  $$\mathcal{B}_P(f,g+hf) = \mathcal{B}_P(f,g)$$ 
  for any $h\in k[x,y]$.
\end{proposition}

\begin{proof}
  Again by Proposition~\ref{prop:aff_transf}, we assume $P = (0,0)$
  and that the $y$-axis is not tangent to $C_f$ nor to $C_g$ at $P$.
  Hence, the strict transforms of curves constructed as
  described in Remark~\ref{rem:isomorphic_blowups} 
  contain all first order infinitely near points of $C_f$ a $C_g$.
  We proceed by the blowup induction, see Remark~\ref{rem:blowup_induction}.

  First, we solve the trivial cases.
  \begin{enumerate}
    \item[(i)] Let $m(f) = m(g) = 1$ 
      with $C_f$ and $C_g$ intersecting transversally in $P = (0,0)$.
      Then the curves $C_f$ and $C_g$ have no infinitely near point in common.
      The same holds for $C_f$ and $C_{g+hf}$.
      To check it, one distinguishes two cases: 
      if $m(h) > 0$,
      then $g+hf$ has the same linear form as $g$, 
      and if $m(h) = 0$, 
      the linear form of $g+hf$ is indeed different from the one of $g$ 
      but again is no multiple of the one of $f$.
      So $\mathcal{B}_P(f,g+hf) = \mathcal{B}_P(f,g) = 1$.
    \item[(ii)] Let $m(f) = 0$, $m(g) \ge 1$, then 
      $\mathcal{B}_P(f,g) = \mathcal{B}_P(f,g + hf) = 0$.
    \item[(iii)] Let $m(f)\ge 1$, $m(g) = 0$, then also $m(g+hf) = 0$
      and again $\mathcal{B}_P(f,g) = \mathcal{B}_P(f,g + hf) = 0$.
  \end{enumerate}

  Now, we use the hypothesis, that the assertion holds for
  $\mathcal{B}_P(f^\prime, (g+hf)^\prime)$, i.e. that
  $$\mathcal{B}_P(f^\prime, (g+hf)^\prime + \tilde h f^\prime) = 
    \mathcal{B}_P(f^\prime, (g+hf)^\prime)$$ 
  for any $\tilde h\in k[x,y]$, and we prove it for
  $\mathcal{B}_P(f,g+hf)$ by case distinction.

  Let $f = F + f_1$, where $F$ is the form consisting 
  of lowest degree terms, so $\deg(F) = m(f)$, 
  and $f_1$ is the polynomial containing the rest of $f$.
  Similarly, $g = G + g_1$ and $h = H + h_1$. 
  For the polynomials locally defining the strict transforms, 
  we have $f^\prime(x,y) = F(1,y) + xf_2(x,y)$
  for some $f_2\in k[x,y]$,
  similarly for $g^\prime$ and $h^\prime$.

  {\em Case 1:} Let $m(f) > m(g)$.
  An easy verification shows that 
  $$(g+hf)^\prime = g^\prime + x^{m(h)+m(f)-m(g)}h^\prime f^\prime.$$
  Since the set of the first order infinitely near points 
  of a curve at the point $P = (0,0)$
  depends only on the form of the lowest degree, 
  those of $C_{g+hf}$ are the same as those of $C_g$.
  So
  \begin{eqnarray*}
    \mathcal{B}_P(f,g+hf) 
      &=& m(f)m(g+hf) + \textstyle{\sum_Q\mathcal{B}_Q(
            f^\prime, (g+hf)^\prime)} \\
      &=& m(f)m(g) + \textstyle{\sum_Q\mathcal{B}_Q(
            f^\prime, g^\prime + x^{m(h)+m(f)-m(g)}h^\prime f^\prime)} \\
      &=& m(f)m(g) + \textstyle{\sum_Q\mathcal{B}_Q(
            f^\prime, g^\prime)} \\
      &=& \mathcal{B}_P(f,g),
  \end{eqnarray*}
  where the sum runs through the first order infinitely near points 
  of $C_f$ at $P$,
  and the third equality follows from the induction hypothesis.

  {\em Case 2:} We assume that $f,g$ 
  is a max-order basis of $(f,g)$ with $m(f)\le m(g)$
  and that $m(hf) \ge m(g)$,
  so $f, g+hf$ is also a max-order basis of $(f,g)$.

  Again, we conclude that $C_g$ and $C_{g+hf}$
  share the same first order infinitely near points at $P$ with $C_f$:
  it is straightforward, if $m(hf) > m(g)$, since in this case 
  $C_g$ and $C_{g+hf}$ have the same first infinitely near points at $P$.
  A little more care is required if $m(hf) = m(g)$:
  here for sure $m(g+hf) = m(g)$ because $f,g$ is a max-order basis 
  for $(f,g)$. 
  Therefore, the form of degree $m(g)=m(f)+m(h)$ in $g+hf$ does not vanish 
  and the $y$-coordinates of the infinitely near points of $g+hf$ are given by 
  the equation $G(1,y) + H(1,y)F(1,y) = 0$. From this we already easily check
  that the first order infinitely near points 
  shared by $C_f$ and $C_{g+fh}$ are the same
  as the first order infinitely near points shared by $C_f$ and $C_g$.

  For the polynomial defining the strict transform of $C_{g+hf}$ 
  the direct computation shows that
  $$(g + hf)^\prime = g^\prime + x^{m(f)+m(h)-m(g)}h^\prime f^\prime.$$
  So we have exactly the same computation as in Case 1 showing that
  $$\mathcal{B}_P(f, g+hf) = \mathcal{B}_P(f, g).$$ 

  {\em Case 3:} We assume that $f,g$ 
  is a max-order basis of $(f,g)$ with $m(f)\le m(g)$
  and that $m(hf) < m(g)$,
  so in this case $f, g+hf$ is not a max-order basis of $(f,g)$.

  In this case, we have that 
  $$(g + hf)^\prime = x^{m(g)-m(f)-m(h)}g^\prime + h^\prime f^\prime$$
  and the set of the first order infinitely near points 
  of $C_{g+hf}$ at $P=(0,0)$ 
  is the union of 
  those of $C_f$ and those of $C_h$, so the following sum 
  goes through the first order infinitely near points of $C_f$.
  \begin{eqnarray}
    \mathcal{B}_P(f, g+hf) 
      &=& m(f)m(g+hf) + \textstyle{\sum_Q\mathcal{B}_Q(
          f^\prime,(g+hf)^\prime)}\notag \\
      &=& m(f)m(hf) + \textstyle{\sum_Q\mathcal{B}_Q(
            f^\prime, x^{m(g)-m(f)-m(h)}g^\prime + h^\prime f^\prime)}\notag \\
      &=& m(f)m(hf) + \textstyle{\sum_Q\mathcal{B}_Q(
            f^\prime, x^{m(g)-m(f)-m(h)}g^\prime)}\label{eq:1} \\
      &=& m(f)m(hf)  
            + \textstyle{\sum_Q(m(g)-m(f)-m(h))\mathcal{B}_Q(f^\prime, x)} 
            + \textstyle{\sum_Q\mathcal{B}_Q(f^\prime, g^\prime)}\label{eq:2} \\
      &=& m(f)m(hf) + (m(g)-m(f)-m(h))m(f) 
            + \textstyle{\sum_Q\mathcal{B}_Q(f^\prime, g^\prime)}\label{eq:3} \\
      &=& m(f)m(g) 
            + \textstyle{\sum_Q\mathcal{B}_Q(f^\prime, g^\prime)}\notag \\
      &=& \mathcal{B}_p(f,g).\label{eq:4}
  \end{eqnarray}
  Here (\ref{eq:1}) follows from the induction hypothesis, 
  (\ref{eq:2}) follows from Proposition~\ref{prop:product}.
  The equality (\ref{eq:3}) follows from the fact that 
  $C_f$ has exactly $m(f)$
  counted with multiplicities infinitely near points at $P$ 
  and none of them coinciding with those of $x$.
  Finally (\ref{eq:4}) follows from the fact that 
  $\mathcal{B}_Q(f^\prime, g^\prime) = 0$ if $Q$ does not belong 
  to the first order infinitely near points of $C_g$.

  {\em Case 4:} We assume that $m(f)\le m(g)$
  and neither $f,g$ nor $f,g+hf$ is a max-order basis of
  the ideal they generate. 

  By Lemma~\ref{lem:mmo_exists}, there is $p\in k[x,y]$ such that
  $f,g+pf$ is a max-order basis and therefore by Case~3
  $$\mathcal{B}_P(f,g) = \mathcal{B}_P(f, g+pf).$$
  On the other hand $g + hf + (p-h)f = g+pf$, so again by Case~3
  $$\mathcal{B}_P(f,g+hf) = \mathcal{B}_P(f, g+pf),$$
  and we get
  $$\mathcal{B}_P(f,g+hf) = \mathcal{B}_P(f, g).$$
\end{proof}

\begin{theorem}\label{thm:main}
  $$\mathcal{B}_P(f,g) = I_P(f,g).$$
\end{theorem}

\begin{proof}
  It is the consequence of the proven propositions 
  and Theorem-Definition~\ref{thm:fulton}.
  Directly from the definition of $\mathcal{B}_P(f,g)$ follows, 
  that the computed number
  satisfies the properties (1), (2), (4) and (5) of Theorem~\ref{thm:fulton}. 
  The remaining properties were verified 
  in Propositions~\ref{prop:aff_transf},~\ref{prop:product},
  and \ref{prop:gen_indep}.
\end{proof}

\section{The Algorithm for computing the local intersection multiplicity}

We implemented our algorithm in Sage for curves given over $\Q$.
During the computations, we need to factorize an univariate polynomial
into linear factors. It might happen that the polynomial does not factor 
over the field we work in at the point, 
and so we make a suitable algebraic extension
to make the factorization possible.

\vskip0.4cm
\noindent
\begin{tabular}{ll}
  \hline
  {\sc Function:} & {\tt IntersectionMultiplicity }\\
  \hline
  {\sc Input:}    & $f,g\in k[x,y]$ representing two curves at $\A^2(k)$,\\
  {\sc Output:}   & intersection multiplicity of $f$ and $g$ 
                    in the point $(0,0)$.\\
\end{tabular}
\begin{enumerate}
  \item {\it // the intersection multiplicity is at least the product of the orders} \\
    $m_f$ := the degree of the lowest term on $f$\\
    $m_g$ := the degree of the lowest term on $g$\\
    $I$ := $m_f.m_g$
  \item
    {\it // take affine charts of the blowups of $f$ and $g$ in $(0,0)$}\\
    $f_1$ := $x^{-m_f}f(x,xy)$ \\
    $g_1$ := $x^{-m_g}g(x,xy)$ 
  \item
    {\it // find all infinitely near points shared by both curves} \\
    $roots := $ roots of $gcd(f_1(0,y),g_1(0,y))$; \\
    (if needed, make an algebraic extension of the base field
    so that $gcd(f_1(0,y),g_1(0,y))$ splits into linear factors;)
  \item
    {\it // run the algorithm for all shared infinitely near points} \\
    {\tt for} $r\in roots$: \\
      $I := I +$ {\tt IntersectionMultiplicity}$(f_1(x, y+r), g_1(x, y+r))$\\
    {\tt endfor}
  \item
    {\it // if the curves $f$ and $g$ are both tangent to $y$-axis at $(0,0)$}\\ 
    {\tt if} $x$ divides the lowest forms of both $f$ and $g$ {\tt then} \\
      {\it // take the other affine charts of the blowup}\\
      $f_1$ := $y^{-m_f}f(xy,y)$ \\
      $g_1$ := $y^{-m_g}g(xy,y)$ \\
      $I := I +$ {\tt IntersectionMultiplicity}$(f_1(x, y), g_1(x, y))$\\
    {\tt end if};
  \item 
    {\tt return} $I$.
\hrule
\end{enumerate}

\section{Examples}

\subsection{Circle and ellipse}
Let $C_f$ be the ellipse given by 
$$f = 5x^2 + 6xy + 5y^2 - 10y$$
and $C_g$ be the circle given by 
$$g = x^2 + (y-1)^2 - 1.$$
In the exposition and in the figures,
we denote 
the polynomial describing the strict transform of $C_f$ resp. $C_g$ 
after the $i$-th blowup by $f_i$ resp. $g_i$. 
The restriction of the blowup morphism %(see Section~\ref{se:blowup}) 
to an affine chart is denoted by $\pi$.
The intersection multiplicity of $C_f$ and $C_g$ in $(0,0)$ 
is found after performing three consecutive blowups, 
see Figure~\ref{fig:1}.

Firstly, both curves are regular in $(0,0)$, so each has only one
infinitely near point of the first order at $(0,0)$. 
We find them by computing the strict transforms of $f$ and $g$
\begin{eqnarray*}
  f_1 &=& 5x + 6xy + 5xy^2- 10y,\\
  g_1 &=& x + xy^2 - 2y.
\end{eqnarray*}
The infinitely near points are the intersections of the strict 
transform with $y$-axis. We see that both curves intersect 
the $y$-axis in the point $(0,0)$. Hence 
$$I_{(0,0)}(f,g) = 1.1 + I_{(0,0)}(f_1,g_1).$$

The point $(0,0)$ is again regular for both $f_1$ and $g_1$.
After the second blowup, we have
\begin{eqnarray*}
  f_2 &=& 5 + 6xy + 5x^2y^2 - 10y,\\
  g_2 &=& 1 + x^2y^2 - 2y
\end{eqnarray*}
and we see that they share the first order
infinitely near point, $(0,1/2)$. So
$$I_{(0,0)}(f_1,g_1) = 1.1 + I_{(0,1/2)}(f_2,g_2).$$

The point $(0,1/2)$ is regular for both $f_2$ and $g_2$
and the curves intersect transversally there.
We detect this in the algorithm after blowing up the curves in the point 
(before doing so we shift the point $(0,1/2)$ to $(0,0)$)
and checking that the strict transforms
\begin{eqnarray*}
  f_3 &=& 5x(xy+1/2)^2 + 6xy + 3 + 5x(xy+1/2)^2 - 10y,\\
  g_3 &=& x(xy+1/2)^2 - 2y
\end{eqnarray*}
have no common point on $y$-axis. 
Therefore
$$I_{(0,1/2)}(f_2,g_2) = 1.1.$$

After summing up, 
the intersection multiplicity of $C_f$ and $C_g$ in $(0,0)$ 
is equal to 3.

\begin{figure}[htb]
  \begin{center}
  \includegraphics[scale=1]{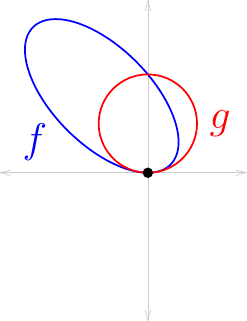}
  \includegraphics[scale=1]{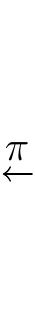}
  \includegraphics[scale=1]{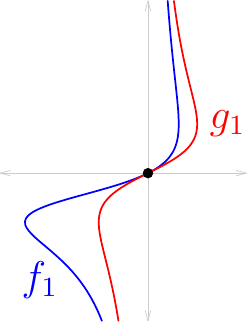}
  \includegraphics[scale=1]{pi_m30_35.pdf}
  \includegraphics[scale=1]{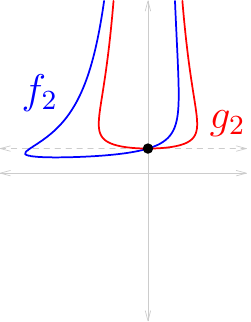}
  \includegraphics[scale=1]{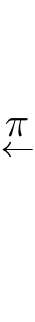}
  \includegraphics[scale=1]{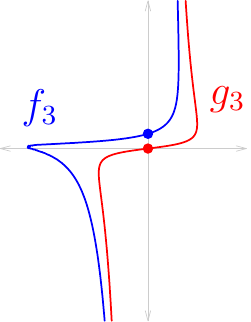}
  \end{center}
  \caption{Computing intersection multiplicity 
  of circle and ellipse in $(0,0)$.}
  \label{fig:1}
\end{figure}

\subsection{Tacnode and ramphoid cusp}

Let $C_f$ be the tacnode curve given by 
$$f = 2x^4 - 3x^2y + y^2 - 2y^3 + y^4$$
and $C_g$ be the curve 
$$g = (x/2)^4 + (x/2)^2y^2 - 2(x/2)^2y - (x/2)y^2 + y^2$$
(the ramphoid cusp).
Again, after performing three consecutive blowups, 
the intersection multiplicity of the curves in $(0,0)$ is found, 
see Figure~\ref{fig:2}.

In this case, the point of intersection 
has multiplicity 2 for both curves. 
Both they have only one first order infinitely near point at $(0,0)$, 
namely $(0,0)$, therefore 
$$I_{(0,0)}(f,g) = 2.2 + I_{(0,0)}(f_1,g_1).$$

For curves $f_1$ and $g_1$, the point $(0,0)$ is again of
multiplicity 2 in both cases. After blowing up in $(0,0)$
we see that the curve $f_1$ has at $(0,0)$ 
two first order infinitely near points: $(0,1)$ and $(0,2)$.
Out of them only $(0,1)$ is shared with $g$, so
$$I_{(0,0)}(f_1,g_1) = 2.2 + I_{(0,1)}(f_2,g_2).$$

Now the curves $f_2$ and $g_2$ are both regular at $(0,1)$
and they intersect transversally (i.e. share no infinitely near point),
so $$I_{(0,1)}(f_2,g_2) = 1.1.$$

Summing up we see that $I_{(0,0)}(f,g) = 9$.

\begin{figure}[htb]
  \begin{center}
  \includegraphics[scale=1]{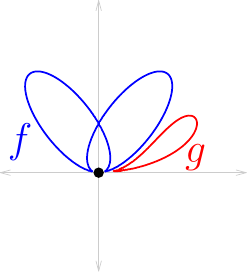}
  \includegraphics[scale=1]{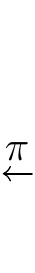}
  \includegraphics[scale=1]{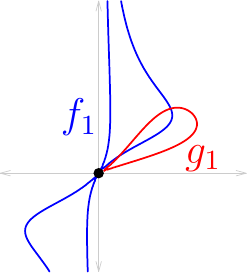}
  \includegraphics[scale=1]{pi_m20_35.pdf}
  \includegraphics[scale=1]{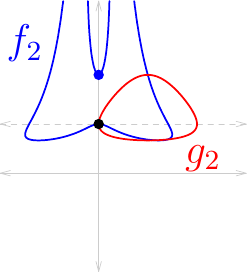}
  \includegraphics[scale=1]{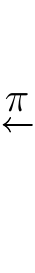}
  \includegraphics[scale=1]{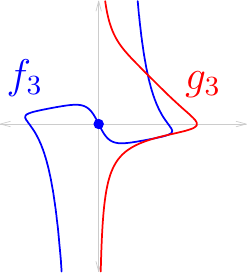}
  \end{center}
  \caption{Computing intersection multiplicity 
  of tacnode and ramphoid cusp in $(0,0)$.}
  \label{fig:2}
\end{figure}

\subsection{Lemniscata of Bernoulli and the four-leaves-curve}

Here, the lemniscata is given by 
$$f = (x^2 + y^2)^2 - (x^2 - y^2)$$
and the four-leaves-curve is given by 
$$g = (x^2+y^2)^3 - (x^2-y^2)^2.$$

In this case the computation of the intersection multiplicity of 
the two curves in $(0,0)$ proceeds
\begin{eqnarray*}
  I_{(0,0)}(f,g) &=& 2.4 + I_{(0,1)}(f_1,g_1) + I_{(0,-1)}(f_1,g_1)\\
  &=& 8 + 1.2 + 1.2 \\
  &=& 12,
\end{eqnarray*}
see Figure~\ref{fig:3}.

\begin{figure}[htb]
  \begin{center}
  \includegraphics[scale=1]{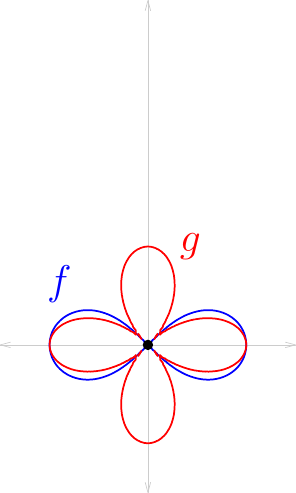}
  \includegraphics[scale=1]{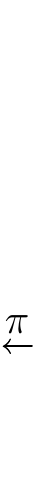}
  \includegraphics[scale=1]{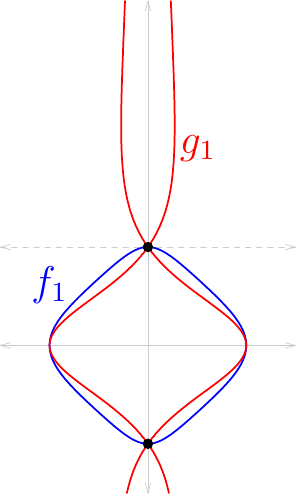}
  \includegraphics[scale=1]{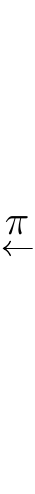}
  \includegraphics[scale=1]{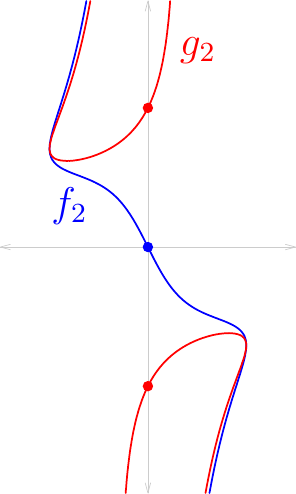}
  \end{center}
  \caption{Computing intersection multiplicity 
  of Bernoulli lemniscata and four-leaves-curve in $(0,0)$.}
  \label{fig:3}
\end{figure}

\section{Performance}
  For comparison we implemented also the algorithm derived from
  axioms for the intersection number given by Fulton.
  When computing the intersection multiplicity of two curves
  in a given point using Fulton's axioms of the intersection
  number, each step in the algorithm in relatively
  simple. The main operation there is finding new generators
  of the ideal $(f,g)$, $f,g\in k[x,y]$, which actually leads to
  a polynomial division with respect to the lexicographic ordering. 
  The drawback of this approach is that the degrees of the polynomials 
  are raising in each step. There are a lot of steps to be executed
  during the computation and there is actually no control
  of their number.

  When computing the intersection multiplicity using blowup,
  we have to
  \begin{itemize}
    \item[--] construct the strict transforms $f_{i+1}$ and $g_{i+1}$ 
      of $f_i$ and $g_i$,
    \item[--] find infinitely near points $Q_j$ shared by $f_{i+1}$ and $g_{i+1}$,
    \item[--] move each $Q_j$ to $(0,0)$.
  \end{itemize}
  in each step. We know in advance that the $i$-th step deals with polynomials 
  of degree $\mathcal{O}(i(\deg f + \deg g))$.
  We also have the upper bound for the number of steps, namely
  there are at most $\deg f\cdot\deg g$ steps executed during the computation.
  This is reflected in much better timing.

  We implemented the algorithm in SageMath 8.3~(\cite{sagemath}) 
  and run it on processor x86\_64, Intel(R) Core(TM) i3-3110M CPU @ 2.40GHz
  with memory 3.7 GiB.

  In Table~\ref{tab:1}, the timings are given for both algorithms.
  They are obtained as the average of 
  10 randomly generated pairs of curves with a given degree 
  and passing through $(0,0)$ with a given multiplicity.
  The intersection number is computed in $(0,0)$.
  The coefficients of the polynomials defining the curves
  are randomly generated integers from $-10$ to $10$.

  \begin{table}
    \begin{center}
      \begin{tabular}{rrrrrrl}
        & deg & m & $I(f,g)$ & axioms & blowup & comment\\
        \hline
         1. & 4   & 1  &  1 & 5.915 ms  & 1.452 ms & transversal intersection \\
         2. & 8   & 1  &  1 & 10.282 ms  & 1.823 ms & transversal intersection \\
         3. & 5   & 2  &  4 & 17.76 ms  & 1.169 ms & transversal intersection \\
         4. & 5   & 3  &  9 & 3407.24 ms  & 1.08 ms & transversal intersection \\
         5. & 6   & 3  &  9 & --        & 1.412 ms & transversal intersection \\
         6. & 15  & 4  & 16 & --        & 4.096 ms & transversal intersection \\
         7. & 5   & 3  & 10 & 346.75 ms & 2.531 ms & a tangent in common \\
         8. & 5   & 3  & 11 & 113.39 ms & 3.372 ms & a double tangent in common \\
         9. & 5   & 2  &  8 & 140.65 ms & 8.107 ms & two tangents in common \\
        10. & 5   & 3  & 13,15 & 186.88 ms & 6.113 ms & two tangents in common \\
        11. & 6   & 3  & 13 & --        & 6.351 ms & two tangents in common \\
        12. & 15  & 3  & 13 & --        & 24.05 ms & two tangents in common \\
        13. & 5   & 2  & 8  & 96.55 ms  & 33.23 ms & tangent cone $x^2+y^2$ in common\\
        \hline
      \end{tabular}
      \caption{``deg'' -- degree of the curves,
        ``m'' -- multiplicity of $(0,0)$ on the curves,
        ``$I(f,g)$'' -- the intersection multiplicity of the curves in $(0,0)$,
         ``axioms'' -- time needed to compute the intersection multiplicity
         using the definition by Fulton,
         ``blowup'' -- time needed to compute the intersection multiplicity
         via blowup,
         ``--'' -- the algorithm did not finish in reasonable time}
      \label{tab:1}  
    \end{center}
  \end{table}

  In some cases (indicated in the table) the algorithm using
  Fulton's axioms for the intersection number did not finish.
  The last row in the table represents the situation where 
  the extension of the field of rationals had to be constructed.
  Apparently this slowed the computation using blowup significantly
  (compare to the 9th row).

% \section{Conclusion}

\section*{Acknowledgment}
  This work was supported by the Slovak Research and Development Agency 
  under the contract No.~APVV-16-0053.

% \appendix

% \section{}
% \label{}

% Bibliographic references with the natbib package:
% Parenthetical: \citep{Bai92} produces (Bailyn 1992).
% Textual: \citet{Bai95} produces Bailyn et al. (1995).
% An affix and part of a reference:
%   \citep[e.g.][Ch. 2]{Bar76}
%   produces (e.g. Barnes et al. 1976, Ch. 2).

%\bibliographystyle{elsarticle-harv}
\bibliographystyle{alpha}
\bibliography{im}

%\begin{thebibliography}{}

% \bibitem[Names(Year)]{label} or \bibitem[Names(Year)Long names]{label}.
% (\harvarditem{Name}{Year}{label} is also supported.)

%\bibitem[]{}

%\end{thebibliography}

\end{document}